\documentclass[11pt]{amsart}

\usepackage{amsmath}
\usepackage{amssymb,amscd}
\usepackage{mathrsfs}       
\usepackage{srcltx}
\usepackage{subcaption}

\usepackage{setspace} 

\usepackage[pdftex]{graphicx}

\usepackage{color}
\usepackage{datetime}
\usepackage{graphicx}
\usepackage{tikz}
\usepackage{tikz-cd}
\usepackage{pgfmath}
\usetikzlibrary{snakes}
\usetikzlibrary{matrix,quotes,arrows.meta}
\usetikzlibrary{arrows,matrix,positioning}
\usetikzlibrary{fit}
\usetikzlibrary{patterns}
\usepackage{arydshln}
\input xy \xyoption{all}

\graphicspath{{images/}}

\numberwithin{equation}{section}

\newtheorem{theorem}{Theorem}[section]

\newtheorem{corollary}[theorem]{Corollary}

\newtheorem{lemma}[theorem]{Lemma}
\newtheorem{prop}[theorem]{Proposition}
\newtheorem{remark}[theorem]{Remark}

\theoremstyle{definition}
\newtheorem{defn}[theorem]{Definition}

\newtheorem{example}[theorem]{Example}
\newtheorem{ques}[theorem]{Question}

\def \begineq{\begin{equation}}
\def \endeq{\end{equation}}

\def \bb{\mathbb}

\def \mc{\mathcal}

\def\oO{\mathcal O}    

\def \CC{{\bb{C}}}

\def \RR{{\bb{R}}}
\def \TT{{\bb{T}}}

\def \({\left(}
\def \){\right)}
\def \<{\langle}
\def \>{\rangle}
\def \bar{\overline}
\def \deg{\mathrm{deg}}

\def\oO{\mathcal O}    
\def\od{\mc{OD}}    

\def \to{\rightarrow}
\def \act{\curvearrowright}

\newcommand{\eqcont}[1]{%
	\mathrel{\ooalign{\hbox{\scalebox{1.5}[1]{$\rhd$}}\cr%
	\kern-0.1ex\raise1.15ex\hbox{\scalebox{1.5}[1]{$\sim$}}\cr%
	\kern1.5ex\raise-1.2ex\hbox{\scalebox{0.7}{#1}}\cr%
	}}}
\begin{document}

\title{Higher equivariant and invariant topological complexity}
\author[M. Bayeh]{Marzieh Bayeh}

\address{Department of Mathematics and Statistics, University of Dalhausie, Halifax, Canada.}

\email{marzieh.bayeh@dal.ca}

\author[S. Sarkar]{Soumen Sarkar}

\address{Department of Mathematics, Indian Institute of Technology Madras, Chennai, 600036, India.}

\email{soumen@iitm.ac.in}



\abstract 
In this paper we introduce the concepts of higher equivariant and invariant
topological complexity; and study their properties. Then we compare them with
equivariant LS-category. We give lower and upper bounds for these new invariants.
We compute some of these invariants for moment angle complexes.

\endabstract

\maketitle

%


\section{Introduction}\label{introsco2}

To estimate the complexity of a configuration space Farber \cite{Far2} introduced
the notion of topological complexity. This invariant of a topological space $X$,
denoted by $TC(X)$, is the least number of open sets that form a covering for
$X \times X$ in which each open set admits a section to the free path fibration 
$$\begin{array}{c}
\pi : X^I \to X \times X  \quad \mbox{defined by} \quad
\pi(\gamma)=\big(\gamma(0),\gamma(1) \big).
\end{array}$$
In particular, $TC(X)$ is the Schwarz genus \cite{Sv} of the map $\pi: X^I \to X \times X$. 

In \cite{Rud} Ruddyak introduced a series of invariants $\{TC_n(X)\}$, called the higher topological complexity. The notion of higher topological complexity is a generalization of topological complexity, as $TC_n(X)$ is related to a motion planning algorithm with $n$ points as input (in addition to the initial and terminal points, some intermediate points are given as well).

When the space $X$ admits an action of a topological group $G$ (for example having a symmetry on the mechanical system or its configuration space), then it is worth considering a motion planning algorithm that is compatible with the action. 
This idea leads us to the equivariant versions of topological complexity. 
Lubawski and Marzantowicz \cite{LM} studied the importance of invariant topological complexity 
when there is a group action on a mechanical system or on the configuration space, and discussed a natural way of thinking about equivariant version of topological complexity (for more details see the introduction of \cite{LM}). 

In this paper we introduce two equivariant versions of the higher topological complexity and study some of their properties.

The first concept is called the higher equivariant topological complexity and it is a generalization of the equivariant topological complexity, $TC_G(X)$, which is introduced by Colman and Grant \cite{CG}. For a $G$-space $X$, Colman and Grant considered the diagonal action of $G$ on $X\times X$. 

The second concept is called the higher invariant topological complexity and it is a generalization of the invariant topological complexity, $TC^G(X)$, which is introduced by Lubawski and Marzantowicz \cite{LM}. For the invariant topological complexity, the product space $X\times X$ has been considered with the
product action of $G \times G$. In \cite{BS2}, the authors compared the equivariant topological complexity with the invariant topological complexity, using the concept of orbit class and orbit diagram.

The paper is organized as follows. 
In Section \ref{Sec:od}, we recall some basic notions
about orbit class, orbit diagram and equivariant LS-category associated to
a group action. We also show the product formula for the product action
under some mild hypothesis.
In Section \ref{Sec:heqtc}, we introduce the higher equivariant topological complexity. 
We also give some lower and upper bounds for this invariant. For some particular cases we show that the equivariant LS-category gives an upper bound for the higher equivariant topological complexity.
In Section \ref{Sec:invTC}, we introduce the higher invariant topological complexity. 
We study some properties of this invariant. 
In particular, we show that if the space has more than one minimal orbit class then this invariant is infinity.
Finally, in the last section we study the equivariant topological complexity of the moment angle complex.



\section{Orbit class and equivariant LS-category}\label{Sec:od}
In this section we recall some results about orbit class, orbit diagram and
equivariant LS-category associated to a group action following \cite{BS2} and \cite{LM}. 
Let $G$ be a compact topological group, acting continuously on a Hausdorff topological 
space $X$. Through out this paper these are the assumptions. In this case $X$ is called a
$G$-space. For each $x \in X$ the orbit of $x$ is denoted by $\oO(x)$, and the isotropy group or stabilizer of $x$ is denoted by $ G_x$.
The orbit space which is equipped with the quotient topology is denoted by $X/G$. 
The fixed point set of $X$ is denoted by $X^G $. 
Here, the fixed point set $X^G$ is endowed with the subspace topology. We denote the closed interval $[0,1]$ in $\RR$ by $I$.

\begin{defn}
Let $X$ be a $G$-space.
\begin{enumerate}
\item A subset $U$ of $X$ is called $G$-invariant subset, if $U$ is stable under the
$G$-action; i.e. $GU\subseteq U$.
\item Let $U$ be a $G$-invariant subset of $X$. The homotopy $H:U \times I \to X$ is called $G$-homotopy
if for any $g \in G$, $x \in U$, and $t \in I$, we have $gH(x,t) = H(gx,t)$.
\end{enumerate}
\end{defn}

\begin{defn}\label{G-cont}
Let $U$ and $A$ be $G$-invariant subsets of a $G$-space $X$. 
We say $U$ is $G$-contractible to $A$ and denote it by 
$$ U \eqcont{G} A  ,$$
if there exists a $G$-homotopy $H: U \times I \to X$ such that 
$H_0$ is the inclusion of $U$ in $X$, and we have 
$H_1(U) \subseteq A$.

If $A$ is an orbit, $U$ is called a $G$-categorical subset of $X$.
\end{defn}
 As a special case of Definition \ref{G-cont}, if $U$ and $A$ are orbits, $U=\oO(x)$
and $A=\oO(y)$, then a $G$-homotopy $H:\oO(x) \eqcont{G} \oO(y)$ is called a
 $G$-path from $\oO(x)$ to $\oO(y)$ \cite[Definition 3.1]{HT}. Note that
 in this case $G_x \leq G_{x_t}$, where $x_t = H(x,t)$. In particular there exists
 $g_0 \in G$ such that $G_x \leq G_{g_0y} = g_{0}G_{y}g_0^{-1}$ (see \cite[Lemma 3.2]{HT}). 

\begin{defn}
Let $X$ be a $G$-space. We say $\oO(x) \sim \oO(y)$ 
if there exist two $G$-paths ${H: \oO(x)\eqcont{G} \oO(y)}$ and $ H': \oO(y) \eqcont{G} \oO(x)$. 
\end{defn}

Note that $\sim$ is an equivalence relation on the set of orbits in $X$ (see \cite{Bay}). We denote the equivalence
class of $\oO(x)$ by $[\oO(x)]$ and call it the orbit class corresponding to $x$.

\begin{defn}
Let $X$ be a $G$-space. On the set of all orbit classes we define the relation $\geq$ as follows:
$$[\oO(y)] \geq [\oO(x)] \quad \text{if} \quad \oO(y) \eqcont{G} \oO(x) .$$
\end{defn}

Here, the relation $\geq$ is independent of the choice of the representative
of the equivalence classes (see \cite{Bay}). Therefore $\geq$ defines a partial order on the
set of orbit classes in $X$. We call the Hasse diagram corresponding to this 
poset an orbit diagram of $X$ and denote it by $\od(G \act X)$. 
See \cite{BS2} and \cite{Bay} for some 
examples of orbit classes and orbit diagrams.


\begin{defn}\label{agls-cat}
Given a $G$-invariant subset $A$ of a $G$-space $X$, an $A\mbox{-}G$-categorical
covering of $X$ is a set of $G$-invariant subsets that form a covering for $X$ and each of which is $G$-contractible to $A$. 
The least value of $n$ for which a $A\mbox{-}G$-categorical
covering $\big\{ U_1 , ... , U_n \big\}$
exists, is called the $A\mbox{-}G$-LS-category of $X$, denoted by $_Acat_G(X)$.
If no such covering exist, we write $_Acat_G(X) = \infty$. 
\end{defn}
This definition is similar to the one in \cite{LM}, but there $A$ is assumed
to be a closed invariant subset of $Y$. 
Note that if the action of $G$ is trivial then $_{pt}cat_G(X)$ is the
classical LS-category $cat(X)$. 
Also if $A $ and $ B$ are two $G$-invariant subsets of $Y$ with $A \subseteq B$, 
then we have 
$$  _B cat_G(Y) \;\; \leq \;\; _A cat_G(Y).$$
Together with this, several properties of $_Acat_G(X)$ have been studied in \cite{LM}.

\begin{defn}\label{equils-cat}
Let $X$ be a $G$-space. A $G$-categorical subset of $X$ is a $G$-invariant subset which is $G$-contractible to an orbit in $X$.
\end{defn}
\begin{defn}\label{equils-cat}
For a $G$-space $X$, a $G$-categorical
covering is a set of $G$-invariant subsets that form a covering 
for $X$ and each of which is a $G$-categorical subset. 
The least value of $n$ for which a $G$-categorical
covering $\big\{ U_1 , ... , U_n \big\}$
exists, is called the equivariant category of $X$, denoted by $cat_G(X)$. 
If no such covering exist, we write $cat_G(X) = \infty$.
\end{defn}

 Although, Definitions \ref{agls-cat} and \ref{equils-cat} may look similar, but
 they are indeed different. For example, $_Acat_G(X)$ satisfies the product formula
(see \cite[Theorem 2.14]{LM}), but $cat_G(X)$ does not in general (see
 \cite[Example 6.4]{BS}).
 

\begin{defn}
Let $G$ be a topological group acting on a topological space $X$. The sequence 
$$ \emptyset=A_0 \varsubsetneq A_1 \varsubsetneq A_2 \varsubsetneq \cdots \varsubsetneq A_n = X $$ 
of open sets in $X$ is called $G$-categorical sequence  of length $n$ if
\begin{itemize}
\item each $A_i$ is $G$-invariant, and
\item for each $1 \leq i \leq n$, there exists a $G$-categorical subset $U_i$ of $X$, such that
$$ A_i - A_{i-1} \subset U_i .$$
\end{itemize}
A $G$-categorical sequence of length $n$ is called minimal if there exists no
 $G$-categorical sequence with smaller length in $X$.
\end{defn}

We recall that in a $G$-space $X$ if $X^K$ is path connected for any closed subgroup $K$ of $G$,
then $X$ is called a $G$-connected space.


\begin{prop}\label{GcatProd}
Let $X_k$ be a $G_k$-connected space for $k=1, 2$ such that $X_1 \times X_2$ 
is completely normal. If $X_k^{G_k} \neq \emptyset$ for $k=1, 2$, then 
$$ cat_{G_1 \times G_2} (X_1\times X_2) \leq cat_{G_1}(X_1) + cat_{G_2}(X_2)-1, $$
where $X_1 \times X_2$ is given the product $G_1 \times G_2$-action.
\end{prop}
\begin{proof}
The idea of proof is similar to the proof for classical category \cite[Theorem 1.37]{CLOT} 
by using a minimal $G$-categorical sequence, 
and the proof is analogous to the proof of \cite[Theorem 2.23]{BS}.

\end{proof}


\begin{lemma}\label{cat_gcat}
If $X$ is a $G$-space with one orbit type, then $cat_G(X) = cat(X/G)$.
\end{lemma}
\begin{proof}
Let $\mathfrak{q} : X \to X/G$ be the orbit map and $U_1, \ldots, U_n$ form a
categorical open cover for $X/G$. Then there is a homotopy 
$H_i : U_i \times I \to X/G$ starting at $U_i$ and contracting to a point in $X/G$. 
By
the hypothesis, the homotopy $H_i$ preserve the orbit structure. So the
Covering Homotopy Theorem of Palais (\cite[II.7.3]{Bre}) implies that there is a
$G$-homotopy $\bar{H} : \mathfrak{q}^{-1}(U_i) \to X$ starting
at $\mathfrak{q}^{-1}(U_i)$ and contracting to an orbit in $X$. 
\begin{center}
\begin{tikzpicture}
\matrix (m) [matrix of math nodes, row sep=3em, column sep=5em]
{ \mathfrak{q}^{-1}(U_i) &  U_i \\
  X & X/G \\};
\path[->, line width=1pt]
(m-1-1) edge node[auto] {$\mathfrak{q}$} (m-1-2)
(m-1-1) edge node[auto] {$ $} (m-2-1)
(m-1-2) edge node[auto] {$ $} (m-2-2)
(m-2-1) edge node[auto] {$\mathfrak{q}$} (m-2-2);
\end{tikzpicture}
\end{center}
Other inequality is clear. This proves the lemma.
\end{proof}
Note that Lemma \ref{cat_gcat} generalizes  \cite[Proposition 3.5]{CG}.

\section{Higher equivariant topological complexity}\label{Sec:heqtc}
In this section we introduce and study the higher equivariant topological
complexity, and compute it for some particular examples. First we recall the 
definition of the equivariant sectional category. This is a generalization
of sectional category for spaces equipped with a $G$-action. 

Let $ f: X \to Y$ be a $G$-map between two $G$-spaces $X$ and $Y$. 
The map $f$ is called a $G$-fibration if it satisfies the homotopy lifting property for $G$-maps, i.e.
for any $G$-space $Z$, a $G$-map $ g_0 : Z \to X $, and any $G$-homotopy $ g' : Z \times I \to Y$ such that
$ f g_0 = g'i_0 $, there exists 
a $G$-homotopy $\tilde{g}: Z \times I \to X$ making the two triangles in the following diagram commute.
\begin{center}
\begin{tikzpicture}
\matrix (m) [matrix of math nodes, row sep=4em, column sep=6em]
{ Z &  X \\
  Z \times I & Y \\};
\path[->, line width=1pt]
(m-1-1) edge node[auto] {$g_0$} (m-1-2)
(m-1-1) edge node[left] {$i_0$} (m-2-1)
(m-2-1) edge node[below] {$g'$} (m-2-2)
(m-1-2) edge node[auto] {$f$} (m-2-2);
\path[->, line width=1pt, dashed](m-2-1) edge node[above] {$\tilde{g}$} (m-1-2);
\end{tikzpicture}
\end{center}
\begin{defn}
The equivariant sectional category of a $G$-fibration $f : X \to Y$, denoted by $secat_G(f)$,  
is the least integer $m$ such that $Y$ can be covered by $m$ invariant open
sets $U_1, \ldots, U_m$, for each of which there exists a $G$-section to $f$,
i.e. there is a $G$-map $s_j \colon  U_j \to X$ such that $f \circ s_j = \iota_{U_j} \colon U_j \hookrightarrow Y$. 
\begin{center}
\begin{tikzpicture}
\matrix (m) [matrix of math nodes, row sep=3em, column sep=5em]
{  &  X \\
  U_j & Y \\};
\path[->, line width=1pt]
(m-2-1) edge[bend left] node[auto] {$s_j$} (m-1-2)
(m-1-2) edge node[auto] {$f$} (m-2-2);
\path[right hook->, line width=1pt](m-2-1) edge (m-2-2);
\end{tikzpicture}
\end{center}
If  no such integer exists then $secat_G(f) = \infty$.
\end{defn}

See Section 4 of \cite{CG} for some basic results on equivariant sectional category.
Note that Since $f$ is a $G$-fibration, if there exists a $G$-map $s_j$ 
making the diagram commute up to $G$-homotopy, i.e. 
$$f \circ s_j \simeq_G \iota_{U_j} \colon U_j \hookrightarrow Y,$$ 
then there exists a $G$-map $s'_j$ making the diagram strictly commute.

Although we are defining the equivariant sectional category only for $G$-fibrations, 
but in fact we can consider it for any $G$-map as follows.
\begin{defn}
Given any $G$-map $f : X \to Y$ with $X$ and $ Y$ path connected $G$-spaces,
a $G$-fibrational substitute of $f$ is defined as a $G$-fibration 
$\hat{f} : E \to Y$ such that there exists  a $G$-homotopy equivalence
$h$ that makes the following diagram of $G$-maps commute.
\begin{center}
\begin{tikzpicture}
\matrix (m) [matrix of math nodes, row sep=3em, column sep=5em]
{  &  E \\
  X & Y \\};
\path[->, line width=1pt]
(m-2-1) edge node[auto] {$h$} (m-1-2)
(m-1-2) edge node[auto] {$\hat{f}$} (m-2-2)
(m-2-1) edge node[below] {$f$} (m-2-2);
\end{tikzpicture}
\end{center}
\end{defn}
\begin{lemma}
Any $G$-map between two path-connected $G$-spaces has a $G$-fibrational substitution.
\end{lemma}
\begin{proof}

Let $f : X \to Y$ be a $G$-map between two path connected $G$-spaces $X$ and $ Y$. 
Considering 
$$ E = X \times_f Y^I  = \big\{\;  (x,\gamma) \; : \; \gamma: I \to Y , \; \gamma(1) = f(x) \;\big\}, $$
with the diagonal action, one can show that the map 
$$  \hat{f} : E \to Y , \quad \hat{f}(x,\gamma) = \gamma(1) $$
 is a $G$-fibration. 

\end{proof} 

Therefore, any $G$-map has a $G$-fibrational substitute. So we can define the equivariant sectional category of any $G$-map $f : X \to Y$ to be the equivariant sectional category of its $G$-fibrational substitute.


\begin{prop}\label{secat1}
For any diagram of $G$-maps $X \xrightarrow{f} Y \xrightarrow{g} Z$, we have 
$$secat_G(gf) \geq secat_G(g) .$$
\end{prop}


Let $Y$ be a $G$-space. Consider the $n$-fold product $Y^n$ with the diagonal action
of $G$. Let $J_n$ be the wedge of $n$ closed intervals $I_i$ for $i = 1, \cdots , n$
where the zero points $0_i \in I_i$ are identified. Then $P_n(Y) =  Y^{J_n}$ is a $G$-space with the following action, 
 $$ G \times P_n(Y) \to P_n(Y),\quad [g \lambda](t) =  g(\lambda(t)).$$

\begin{lemma}\label{en_fib}
Let $e_n \colon P_n(Y) \to Y^{n} $ be a $G$-map defined by
$$  e_n(\lambda) = \big(\lambda(1_1), \ldots, \lambda(1_n)\big) . $$ 
Then $e_n $  is a $G$-fibration.
\end{lemma}

To prove Lemma \ref{en_fib} we need the following result.
Let $f: X \to Y$ be a $G$-map between two $G$-spaces $X$ and $Y$. 
Let $M_f$ be the mapping cylinder of  $f$ and consider $h: M_f \to Y \times I $ defined by
$$   h(x,t) = \big( f(x), t \big) , \quad h(y,0) = (y,0) . $$
A retracting function for $f$ is a map $\rho : Y \times I \to M_f$ which is a left inverse of $h$. 
One can show that if there exists a retracting function for $f:X \to Y$, where $X$ and $Y$ are locally compact Hausdorff spaces, then for any $G$-space $Z$, the $G$-map 
$$ \zeta_f: Z^Y \to Z^X $$ 
defined by $\zeta_f(u) = u \circ f$ is a $G$-fibration.
Using this idea we can prove Lemma \ref{en_fib} as follows.

\begin{proof}[Proof of Lemma \ref{en_fib}]
Note that the inclusion $ B= \{1_1, 1_2, \cdots, 1_n \} \subset J_n$ has a retracting function and $Y^B$ is homeomorphic to $Y^{n}$. Therefore, $e_n $  is a $G$-fibration.
\end{proof}

\begin{defn}
The higher equivariant topological complexity, denoted by $TC_{G, n}(Y)$, is defined by $$TC_{G, n}(Y)= secat_G(e_n).$$

\end{defn}

When $n=2$ in the above definition, then $TC_{G, n}(Y)$ is the equivariant topological complexity defined in \cite{CG}. We also remark that if $n=2$ and $G$ acts trivially or in particular $G$ is trivial, then $TC_{G, n}(Y)$ is the Farber's complexity of a motion planning algorithm on $Y$. If $n > 2$ and $G$ is trivial then $ TC_{G, n}(Y)$ is Rudyak's 
higher topological complexity \cite{Rud}. In the following we give some equivalent definition of higher equivariant topological complexity. 

\begin{prop}\label{equiv_defn_hetc}
For a $G$-space $Y$, the following statements are equivalent:
\begin{enumerate}
\item $_{\Delta_n(Y)}cat_{G}(Y^n) \leq k$;

\item $TC_{G,n}(Y) \leq k$;

\item there exist $ k$ invariant open sets $V_1, \ldots, V_k$ which cover $Y^n$
and for each open set $V_j$ there exists a map $s_j : V_j \to P_n(Y)$
such that the map $e_n \circ s_j$ is $G$-homotopic to the inclusion
$V_j \hookrightarrow Y^n$. \label{hom_equ}
\end{enumerate}
\end{prop}
\begin{proof}
The proof is similar to the proof of  \cite[Lemma 3.5]{LM} with suitable changes
in domain and co-domain of the respective maps. 
\end{proof}

Note that since $e_n$ is a $G$-fibration, in statement (\ref{hom_equ}) if there exists a map $s_j : V_j \to P_n(Y)$
such that the map $e_n \circ s_j$ is $G$-homotopic to the inclusion
$V_j \hookrightarrow Y^n$, then there exists a map $s'_j $ 
such that the map $e_n \circ s'_j$ is equal to the inclusion
$V_j \hookrightarrow Y^n$.
\begin{corollary}
If $Y$ is a $G$-space then $TC_n(Y) \leq TC_{G, n}(Y)$.
\end{corollary}

Let $Y$ and $Z$ be $G$ spaces. Then $Z$ is $G$-dominated by $Y$ if there exist $G$-maps $f : Y \to Z$ and $g : Z \to Y $ such that $f \circ g \simeq_{G} Id_Z $. 
In addition, if $g \circ f \simeq_{G} Id_Y$ then $f$ and $g$ are called $G$-homotopy equivalences, as
well as $Y$ and $Z$ are called $G$-homotopy equivalent.


\begin{prop}
Higher equivariant topological complexity is a $G$-homotopy invariant.
\end{prop}
\begin{proof}
This follows from Proposition \ref{equiv_defn_hetc} and \cite[Proposition 2.4]{LM}. 
\end{proof}


Consider the diagonal map $\bigtriangleup_n : Y \to Y^n$ defined by
 $\bigtriangleup_n(y)= (y, \ldots, y)$. We have the following result.


\begin{prop}\label{diag_1}
Let $U$ be a $G$-invariant open subset of $Y^n$. There exists a $G$-section $ s: U \to P_n(Y)$ to the $G$-fibration $e_n : P_n(Y) \mapsto Y^{n}$ if and only if the inclusion $\iota : U \to Y^{n}$ is $G$-homotopic to a map with values in the diagonal $\bigtriangleup_n(Y) \subseteq Y^{n}$. 
\end{prop}
\begin{proof}
Consider the $G$-map $\varphi : Y \hookrightarrow P_n(Y)$ 
defined by $y \mapsto c_y$ where $c_y : J_n \to Y$ is the constant map at $y$. The result follows from the fact that $\varphi $ is a $G$-homotopy equivalent.
\end{proof}

\begin{corollary}
Let $Y$ be a $G$-space. Then 
$$TC_{G, n}(Y) \;\; = \;\;  secat_G(\bigtriangleup_n) \;\; = \;\;
_{\bigtriangleup_n (Y)}cat_{G} Y^n .$$
\end{corollary}

\begin{prop}
$TC_{G, n}(Y) \leq TC_{G, n+1}(Y)$ for all $n \geq 1$.
\end{prop}
\begin{proof}
The natural inclusion $\iota_n \colon J_n \hookrightarrow J_{n+1} $
induces a surjective continuous $G$-map $$f_n \colon P_{n+1}(Y) \to P_n(Y)$$ defined by 
$\alpha \mapsto \alpha \circ \iota_n$. 
On the other hand we have a continuous $G$-map
$$\mathfrak{y}_n \colon Y^n \to Y^{n+1}$$ 
defined by $(y_1, \ldots, y_n) \mapsto
(y_1, \ldots, y_n, y_n)$. Let $V \subset Y^{n+1}$ be a $G$-invariant open subset
such that there is a $G$-map $s \colon V \to P_{n+1}(Y)$ with $e_{n+1} \circ s \simeq_G id_V$.
Then $U = \mathfrak{y}_n^{-1}(V)$ is a $G$-invariant open subset of $Y^n$. 
\begin{center}
\begin{tikzpicture}
\matrix (m) [matrix of math nodes, row sep=3em, column sep=5em]
{  &  V &  P_{n+1}(Y) & P_n{Y} \\
  U &   &  &  Y^n \\};
\path[->, line width=1pt]
(m-2-1) edge node[auto] {$\mathfrak{y}_n$} (m-1-2)
(m-1-2) edge node[auto] {$s$} (m-1-3)
(m-1-3) edge node[auto] {$f_n$} (m-1-4)
(m-1-4) edge node[right] {$e_n$} (m-2-4);
\path[right hook->, line width=1pt, dashed](m-2-1) edge (m-2-4);
\end{tikzpicture}
\end{center}
Then the map 
$ f_n \circ s \circ \mathfrak{y}_n = \rho  : U \to P_n{Y}$
is a $G$-homotopy section. This proves the proposition.

\end{proof}

Let $EG \mapsto BG$ be the universal principal $G$-bundle and $Y_G = EG \times_{G} Y$ be the orbit space of the diagonal $G$-action on $EG \times Y$. The space $Y_G$ is known as Borel space of the $G$-space $X$ and $H_G^{\ast}(Y)=H^{\ast}(Y_G)$ is called Borel equivariant cohomology of $Y$. Here the coefficients of the cohomology rings are in a filed.

\begin{prop}
If there exist cohomological classes $\alpha_1, \ldots, \alpha_k \in H_G^{\ast}(Y^{n+1})$ such that $0 = {\bigtriangleup_n}_{G}^{\ast}(\alpha_j) \in H_G^{\ast}(Y)$ for all $j$ and the product $\alpha_1 \cdots \alpha_k$ is non-zero, then $TC_{G, n}(Y) \geq k$.
\end{prop}
\begin{proof}
The proof is analogous to the proof of \cite[Theorem 5.15]{CG}.
\end{proof}


From the definition of higher equivariant topological complexity, for any $G$-space $Y$ 
and for any $n \geq 1$, one have 
$$ TC_{n}(Y) \leq TC_{G, n}(Y) .$$
Also note that if $Y$ is not $G$-connected, then since for all $n \geq 2$ we have $TC(Y) \leq TC_n(Y)$, we obtain that
$$ TC_{G, n}(Y) = \infty .$$ 

The following proposition shows the relations among the topological complexity 
of the fixed point sets under the action of different subgroups of $G$. 
The proof of each statement can be obtained from the proof of 
\cite[Proposition 5.3]{CG} after a modification in domain and co-domain of the
respective maps.
\begin{prop}\label{rel_tc_subgp}
For a $G$-space $Y$, let $H$ and $ K$ be closed subgroups of $G$. Then
\begin{enumerate}
\item $TC_{K, n}(Y^H) \leq TC_{G, n}(Y)$ if $Y^H$ is $K$-invariant.

\item $TC_{n}(Y^H) \leq TC_{G, n}(Y)$, in particular $TC_{n}(Y) \leq TC_{G, n}(Y)$.

\item $TC_{K, n}(Y) \leq TC_{G, n}(Y)$.
\end{enumerate}
\end{prop}


In a similar spirit of a general problem mentioned in \cite{Rud}, one may ask the following. 
\begin{ques}
Given a non-decreasing sequence $\{a_n\}$ of natural number can one construct a topological $G$-space $Y$ such that $TC_{G, n}(Y)=a_n$ for some group $G$?
\end{ques}
\begin{prop}\label{free_act}
Let $G$ be a connected paracompact Hausdorff topological space acting freely on itself. Then $TC_{G, n}(G) =cat(G^{n-1})$.
\end{prop}

\begin{proof}
First we remark that $G^n/G$ is homeomorphic to $G^{n-1}$ where the $G$-action on $G^n$ is diagonal. 
The homeomorphism is given by 
$$[g_1, g_2, \ldots, g_n] \mapsto (g_1^{-1}g_2, \ldots, g_1^{-1}g_n) .$$
Let ${\bf g}=(g_1, \ldots, g_n) \in G^n$ and ${\bf h}=(h_1, \ldots, h_n) \in G^n$. 
Let $\pi \colon G^n \to G^n/G$ be the orbit map. Since $G$ is connected, 
there exists a path $\alpha \colon I \to G^n$ such that $\alpha(0)={\bf g}$ and $\alpha(1) = {\bf h}$. 
Define $H \colon \oO({\bf g}) \times I \to G^n  $ by
$$ \Big( g (g_1, \ldots, g_n) , t \Big) \mapsto g \alpha(t).$$ 
Then $H$ is a $G$-path from the orbit $\oO({\bf g})$ to the orbit $\oO({\bf h})$. 
So the two orbits in the $G$ space $G^n$ are $G$-homotopic. 
In particular, any orbit in the $G$-space $G^n$ is $G$-contractible to the diagonal 
$\bigtriangleup_n(G) \subset G^n.$ 
Therefore, we obtain that  
$$  cat_G(G^n) \;\; = \;\; _{\bigtriangleup_n(G)} cat_G(G^n) \;\; = \;\;  TC_{G, n}(G) .$$

Let $\{U_1, \ldots, U_k\}$ be the $G$-categorical cover of $G^n$. 
Since $\pi$ is an open map, then $\{\pi(U_1), \ldots, \pi(U_k)\}$ is a categorical open cover of $G^n/G$. 
So we have
$$ cat(G^n/G) \leq cat_G(G^n) .$$

On the other hand, since $G$ is paracompact and Hausdorff (so is $G^n$), the map $\pi$ is a principal $G$-bundle (see \cite[Theorem II.5.8]{Bre}). So if the open subset $V$ is  contractible to a point $x$ in $G^n/G$, then $\pi^{-1}(V)$ is equivariantly homeomorphic to $V \times G$. Hence $\pi^{-1}(V)$ is a $G$-catagorical open subset of $G^n$. This implies that $cat_G(G^n) \leq cat(G^n/G)$. 
Therefore, we obtain $TC_{G,n}(G) = cat(G^{n-1})$. 

\end{proof}

We remark that when $n=2$, the Proposition \ref{free_act} is the same as
 \cite[Theorem 5.11]{CG}. However, the proof in this paper is different
 and it is for all $n$. 


\begin{prop}\label{rel_tc_catg}
If $X$ is $G$-connected, then $TC_{G,n}(X) \leq cat_G(X^n)$. In addition, if
$X^G \neq \emptyset$ then $TC_{G, n} (X) \leq n~ cat_G(X) -1$.
\end{prop}
\begin{proof}
The first part follows from Proposition 4.5 and \cite[Proposition 5.6]{CG}.
The second part follows from \cite[Theorem 2.23]{BS}.
\end{proof}


\begin{prop}
Let $X$ be a $G$-connected space such that $X^{G} \neq \emptyset$. Then $$TC_{G,n}(X)
\leq n TC_{G, 2} -1.$$
\end{prop}
 
\begin{proof}
This follows from Proposition \ref{rel_tc_catg} and from the fact that $TC_{G,2}(X) \geq cat_G(X)$.
\end{proof}


\begin{prop}
Let $X$ be a $G$-connected topological group such that $G$ acts on $X$ by
topological group homomorphism. Then $TC_{G, 2}(X) = cat(X/G)$.
\end{prop}
\begin{proof}
This follows from Lemma \ref{cat_gcat} and \cite[Proposition 5.12]{CG}.
\end{proof}

\begin{example}\label{eg_tc_s1}
Consider $S^3 = \{(z_1, z_2) \in \CC^2 ~|~ |z_1|^2 + |z_2|^2 = 1\}$ with the
$S^1$-action $ S^1 \times S^3 \to S^3 $ defined by
$$ (\alpha, (z_1, z_2)) \to (\alpha z_1, z_2).$$ 
With this action $S^3$ is $S^1$-connected. Also we have 
$$(S^3)^{S^1} = \Big\{ \; (0, z_2) \in S^3 \; \Big\} \cong S^1 .$$ 
Thus by \cite[Theorem 3.3]{BS2}, we have $cat_{S^1}(S^3) \geq 2$.
Let $(0, x), (0, y) \in (S^3)^{S^1} \subset S^3$ with $x \neq y$. Then $S^3 - \{ (0, x) \}$
and $S^3 - \{ (0, y) \} $ are $S^1$-categorical subsets which form a covering for $S^3$. So we have $cat_{S^1}(S^3) 
\leq 2$. Thus 
$$  cat_{S^1}(S^3)=2 .$$ 
Therefore, from the results in Section 4 of \cite{Rud}, Proposition
\ref{rel_tc_subgp} and Proposition \ref{rel_tc_catg} we obtain that
$$n \leq TC_{S^1, n} (S^3) \leq 2n -1.$$
\end{example}

\section{Higher invariant topological complexity}\label{Sec:invTC}
In this section we introduce and study the higher invariant topological complexity,
and discuss the connections between the higher equivariant and invariant topological
complexity. Moreover, we compute these two invariants for some particular spaces. 


The following definition is a particular case of Definition \ref{agls-cat} and the motivation behind this definition can be found in the introduction of \cite{LM}.
\begin{defn}
Let $G$ be a topological group and $Y$ be a $G$-space. Let $\daleth_n(Y)$ be the saturation of the diagonal $\Delta_n(Y) \subset Y^n$ with respect to the $G^n$-action,
$$ \daleth_n(Y) = G^n  \cdot \Delta_n(Y) \subset Y^n. $$
We define the higher invariant topological complexity of $Y$ as the following,
$$ TC^{G, n}(Y) = _{\daleth_n (Y)}cat_{G^n} Y^n.$$
\end{defn}
When $n=2$, then $TC^{G, 2}(Y)$ is the invariant topological complexity as in \cite{LM}. There exists an equivalent definition of higher invariant topological complexity similar to the idea in \cite[Lemma 3.8]{LM}. Note that the higher invariant topological complexity is a $G$-homotopy invariant.


Let $Y$ be a $G$-space. Define 
$$(PY)^n_{Y/G} = \Big\{ (\alpha_1, \ldots, \alpha_n) \in (PY)^n  \; : \;  G \cdot\alpha_i(0) = 
G \cdot \alpha_j(0) \;\; \mbox{for}\;\;  1 \leq i, j \leq n  \Big\}.$$
Note that $(PY)^n_{Y/G}$ is a $G^n$-space under the following action 
$$(g_1, \ldots, g_n) \cdot (\alpha_1, \ldots, \alpha_n)= (g_1\alpha_1, \ldots, g_n\alpha_n).$$ 
Therefore, the map 
$$\mathfrak{q}_n: (PY)^n_{Y/G} \to Y^n $$
defined by $\mathfrak{q}_n(\alpha_1, 
\ldots, \alpha_n) = \big( \alpha_1(1), \ldots, \alpha_n(1) \big)$ is a $G^n$-map.


\begin{prop}\label{Gn_fib}
The map $\mathfrak{q}_n$ is a $G^n$-fibration.
\end{prop}
\begin{proof}
Consider the map 
$$\bar{\mathfrak{q}}_n: (Y^n \times Y^n)^I \to Y^{n} \times Y^n$$
defined by $\bar{\mathfrak{q}}_n(\alpha)=(\alpha(0), \alpha(1)).$ Note that
$$\bar{\mathfrak{q}}_n^{-1}(\daleth_n(Y) \times Y^n)= (PY)^n_{Y/G} .$$
 One can show that $\mathfrak{q}_n= pr \circ \bar{\mathfrak{q}}_n$ restricted on
 $\bar{\mathfrak{q}}_n^{-1}(\daleth_n(Y) \times Y^n)$. Then the proof is analogous
 to the proof of \cite[Proposition 3.7]{LM}.
\end{proof}

Note that for $n=2$, the Proposition \ref{Gn_fib} is the same as  \cite[Proposition 3.7]{LM}.

\begin{lemma}
Let $Y$ be a $G$-space and $n \geq 2$. Then the following statements are equivalent:
\begin{enumerate}
\item $TC^{G, n}(Y) \leq k$;

\item $_{\daleth_n (Y)}cat_{G^n} Y^n \leq k$;

\item There exist $k$-many $G^n$-invariant open subsets $V_1, \ldots, V_k$ which form a covering for 
$Y^n$ and for each of which there exists a $G^n$-map $\beta_j : V_j \to (PY)^n_{Y/G}$ such that $\mathfrak{q}_n \circ 
\beta_j = \iota_j  : V_j \hookrightarrow Y^n$ for $j=1, \ldots, k$.

\item There exist $k$-many $G^n$-invariant open subsets $V_1, \ldots, V_k$ which form a covering for 
$Y^n$ and for each of which there exists a $G^n$-map $\bar{\beta}_j : V_j \to (PY)^n_{Y/G}$ such that $\mathfrak{q}_n \circ 
\bar{\beta}_j $ is $G^n$-homotopic to $ \iota_j : V_j \hookrightarrow Y^n$ for $j=1, \ldots, k$.
\end{enumerate}
\end{lemma}

\begin{proof}
The proof is similar to the proof of  \cite[Lemma 3.8]{LM} with some
modification in the spaces and maps.
\end{proof}



\begin{prop}
Let $Y$ be a $G$-space. We have $TC_n(Y^G) \leq TC^{G, n} (Y) $ for $n \geq 2$.
\end{prop}
\begin{proof}
This is similar to the proof of  \cite[Corollary 3.26]{LM}.
\end{proof}


\begin{lemma}\label{oeq}
Let $Y$ be a $G$-space. Then $(y_1, \ldots, y_n) \in \daleth_n(Y)$ if and only if for $ 1 \leq i , j \leq n $ we have $\oO(y_i) = \oO(y_j)$.
\end{lemma}
\begin{proof}
Let $(y_1, \ldots, y_n) \in \daleth_n(Y)$, so there exist $y \in Y$, 
and $g_j \in G$ for $ 1 \leq j \leq n $ such that $(y_1, \ldots, y_n) = (g_1y, \ldots, g_ny)$. 
Therefore, we obtain that $\oO(y_j) = \oO(y)=\oO(y_i)$ for $ 1 \leq i,j \leq n $.

Conversely, let $\oO(y_i) = \oO(y_j)$ for $ 1 \leq i,j \leq n $. 
Then for each $j \in \{1, \ldots, n\}$, there exists a $g_j \in G$ such that $y_j = g_j y_1$. So we obtain that 
$$ (y_1, \ldots, y_n) \in \daleth_n(Y). $$
\end{proof}


\begin{lemma}\label{not_int_dia}
If $Y$ has more than one minimal orbit class, then $\daleth_n(Y)$ does not intersect all minimal
orbit classes of the $ G^n $-space $Y^n$.
\end{lemma}
\begin{proof}
Assume that $[\oO(m)]$ and $ [\oO(n)]$ are two distinct minimal orbit classes of $Y$. 
Then by  \cite[Proposition 2.20]{BS2}, the orbit class $\big[\oO(m) \times \oO(n)^{n-1}\big]$ is a minimal orbit class in $Y^n$. If $\daleth_n(Y)$ intersect $\big[\oO(m) \times \oO(n)^{n-1}\big]$, then there exists $u, v_1, \ldots, v_{n-1} \in Y$ such that
$$ \oO(u) \times \oO(v_1) \times \cdots \times \oO(v_{n-1}) \in \big[\oO(m) \times \oO(n)^{n-1} \big] $$ 
and 
$$ \oO(u) \times \oO(v_1) \times \cdots \times \oO(v_{n-1}) \subset \daleth_n(Y). $$
Hence there exist $y \in Y$ and $g, h \in G$ such that
$u=gy \in [\oO(m)]$ and $v_1=hy \in [\oO(n)]$. 
Thus $ \big[\oO(u) \big] = \big[\oO(v_1) \big]$, which implies that $ \big[\oO(m) \big] = \big[ \oO(n) \big]$ and contradicts the assumption. 
\end{proof}

\begin{theorem}\label{comp_tcg}
If $Y$ has more than one minimal orbit class, then $TC^{G, n}(Y) = \infty$. 
\end{theorem}
\begin{proof}
Since $\daleth_n(Y)$ does not intersect all minimal orbit classes of the
 $G^n$-space $Y^n$, by  \cite[Theorem 4.7]{BS2} we have
$$ TC^{G, n}(Y) \; = \;  _{\daleth(Y)}cat_{G^n}(Y^n) \; = \; \infty.$$
\end{proof}

Even though it seems $ TC^{G, n}(Y) $ is meaningful on a specific category, it satisfies our natural expectation, Proposition \ref{prop:g-acts-free}, which is not true for higher equivariant topological complexity. We note that Lemma \ref{not_int_dia} and Theorem
\ref{comp_tcg} are proved in \cite[Section 4]{BS2} for $n=2$. 

\begin{prop}
$TC^{G, n}(Y) \leq TC^{G, n+1}(Y)$ for all $n \geq 2$.
\end{prop}
\begin{proof}
Consider $G^n$ as the subgroup $G^n \times \{e\}$ of $G^{n+1}= G^n \times G$. 
Then the map $$\xi_n : (PY)^{n+1}_{Y/G} \to (PY)^n_{Y/G}$$
defined by $(\alpha_1, \ldots, \alpha_n, \alpha_{n+1}) \to (\alpha_1, \ldots, \alpha_n)$
is a surjective continuous $G^n$-map. On the other hand we have a continuous $G^n$-map
$$\phi_n \colon Y^n \to Y^{n+1}$$ defined by $(y_1, \ldots, y_n) \mapsto
(y_1, \ldots, y_n, y_n)$. Let $V \subset Y^{n+1}$ be a $G^{n+1}$-invariant open subset
such that there is a $G^{n+1}$-map 
$$ s \colon V \to (PY)^{n+1}_{Y/G} $$ 
with
$\mathfrak{q}_{n+1} \circ s \simeq_{G^{n+1}} id_V$. So $U = \phi_n^{-1}(V)$
is a $G^n$-invariant open subset of $Y^n$. Then the map
$$U \xrightarrow{\phi_n} V \xrightarrow{s} (PY)^{n+1}_{Y/G} \xrightarrow{\xi_n} 
(PY)^n_{Y/G}$$ is a $G^n$-homotopy section. This proves the Proposition.

\end{proof}

\begin{prop}\label{prop:g-acts-free}
If $G$ acts freely on $Y$, then $ TC^{G, n}(Y) = TC_{n}(Y/G)$ for $n \geq 2$.
\end{prop}
\begin{proof}
The orbit space of the $G^n$-action on $Y^n$ is $(Y/G)^n$. Thus for a
$G^n$-invariant open subset $V$ in $Y^n$, the quotient space $V/G^n$ is an open subset in $(Y/G)^n$. 
Also we have the following commutative diagrams of surjective continuous maps,
$$
\begin{CD}
(PY)^G_n @>\mathfrak{q}_n>> Y^n\\
@VVV @VVV\\
P_nY @>e_n>> (Y/G)^n
\end{CD}
$$
where the vertical arrows are orbit maps. Note that any $G^n$-invariant local section
for $\mathfrak{q}_n$ descends to a local section for $e_n$. Therefore, we have
$$TC_{n}(Y/G) \leq TC^{G, n}(Y) .$$ 
By analogy to the proof of  \cite[Theorem 3.10]{LM},  
 one can show the other inequality. 
\end{proof}

\begin{remark}
In \cite{BK} the authors introduced a topological invariant called effective 
topological complexity. Their sequence of effective topological complexities
decreases and Farber's topological complexity is an upper bound for that. But the 
sequences of the higher equivariant topological complexity and higher invariant topological complexity 
that we introduce in this paper
are both increasing sequences, and sometimes they are strictly increasing. 
\end{remark}

\begin{theorem}
The inclusion $ \daleth_n(X) \subset X^n $ is a $G^n$-cofibration if
$G$ is a finite group.
\end{theorem}

\begin{proof}
Note that
\begin{align*}
\daleth_n(X) &= \Big\{ (g_1x, g_2x, \ldots, g_nx) \in X^n \; : \;  (g_1, g_2, \ldots, g_n) \in G^n \Big\}\\
&= \Big\{ (x, g_2x, \ldots, g_n x) \in X^n \; : \;  (g_2, \ldots, g_{n}) \in G^{n-1} \Big\}.
\end{align*}
Let $h=(h_1, h_2, \ldots, h_n) \in G^n$ and $\<h\>$ be the cyclic subgroup of $G$ generated by
$h$. 
Then we have $h(x, g_2x, \ldots, g_nx) = (x, g_2x, \ldots, g_nx) $ if and only if for $i=2, \ldots, n$ we have 
$$ x \in X^{\<h\>} \quad\mbox{and} \quad x \in X^{\<g_i^{-1}h_ig_i\>}  .$$ 
The result follows by using a similar argument as in the proof of  \cite[Theorem 3.15]{LM}, 
and the result is same when $n=2$.
\end{proof}

\begin{corollary}
Let $G$ be a compact abelian topological group and $Y$ be a compact $G$-ANR
such that for any closed subgroup $H$ there is a finite subgroup $H_G$ of $G$
satisfying $X^H=X^{H_G}$.
Then the inclusion $\daleth_n(Y) \subset Y^n $ is a $G^n$-cofibration. 
\end{corollary}

\begin{prop}\label{prop2}
Let $Y$ be a $G$-space. Then we have 
$$ TC^{G,n}(Y) \; \leq \;  _{A}cat_{G^n}X^n,$$
where $ A = \oO(y)^n$ for some $y \in Y$.
\end{prop}
\begin{proof}
Since $(y, \ldots, y) \in \bigtriangleup_n(Y)$, we obtain that
 $\oO(y)^n \subseteq \daleth_n(Y)$.
The result follows from  \cite[Lemma 2.13 (1)]{LM}.
\end{proof}
We remark that Proposition \ref{prop2} extends \cite[Proposition 3.23]{LM}
and \cite[Proposition 4.9]{BS2}

\begin{prop}
Let $Y$ and $Z$ be a $G$- and $K$-space, respectively. 
If $\daleth_n(Y) \subset Y^n$ is a $G^n$-cofibration  and 
$\daleth_n(Z) \subset Z^n$ is a $K^n$-cofibration, then $$TC^{G \times K, n}(Y \times Z) \leq
TC^{G, n}(Y) + TC^{K, n}(Z) -1.$$ 
\end{prop}
\begin{proof}
Note that $\daleth_n(Y \times Z) = \daleth_n(Y) \times \daleth_n(Z)$. 
So by  \cite[Corollary 2.16]{LM}, the result follows.
\end{proof}

\begin{example}
We adhere the notation of Example \ref{eg_tc_s1}. Let $A= (0, x)^n$. Then applying 
Proposition \ref{prop2} we have $TC^{S^1, n}(S^3) \leq 2^n$. 
\end{example}


\section{Examples}
In this section we compute the equivariant LS-category and give some bounds
for the equivariant topological complexity of moment angel complexes. Moment angle
complexes, which are special type of polyhedral product, are the center of
interest in toric topology. Several properties of moment angle complexes can be
found in \cite{BBCG, BP, DJ}. 
Moreover, the computation of LS-category and topological complexity of moment angle 
complexes are given in \cite{BG, Kam}. 
 
Let $K$ be a simplicial complex on $[m]=\{1, \ldots, m\}$ vertices. For each
simplex $\sigma \in K$, we define 
$$(D^2, S^1)^{\sigma} = \Big\{ (x_1, \ldots, x_m)
\in (D^2)^m ~~: ~~ x_i \in S^1 = \partial{D^2} ~ \mbox{when}~ i \notin \sigma  \Big\}.$$
The polyhedral product 
\begin{equation}\label{mont_ang}
\mathcal{Z}_K = \bigcup_{\sigma \in K} (D^2, S^1)^{\sigma} \subset (D^2)^m
\end{equation}
 is called the moment angle complex of $K$. The space $\mathcal{Z}_K$
has a natural $T^m = (S^1)^m$ action and is a manifold if $K$ is a triangulated
sphere, see \cite[Lemma 6.13]{BP}.

\begin{prop}
If $S$ be the set of all maximal simplices of $K$, then $cat_{T^m}(\mathcal{Z}_K)=|S|$.
\end{prop}
\begin{proof}
Note that if $\tau$ is a face of $\sigma$ in $K$, then $ (D^2, S^1)^{\tau} 
\subseteq (D^2, S^1)^{\sigma}$. So we have 
$$\mathcal{Z}_K = \bigcup_{\sigma \in S} (D^2, S^1)^{\sigma} \subset (D^2)^m.$$ 
The topology on $\mathcal{Z}_K$ is the subspace
topology of $(D^2)^m$. Also any simplex of $K$ is a face of a maximal simplex.
So the set 
$$ \Big\{ (D^2, S^1)^{\sigma} \; : \;  \sigma \in S  \Big\} $$ 
is an open covering for $\mathcal{Z}_K$. Moreover, $(D^2, S^1)^{\sigma}$ is a $T^m$-invariant subset
which is equivariantly contractible to the orbit $(S^1)^{\sigma}$ in $\mathcal{Z}_K$ where
\begin{equation}\label{s1sigma}
(S^1)^{\sigma}=\Big\{ (x_1, \ldots, x_m) \in \mathcal{Z}_K \; : \;  
x_i=0 \;\mbox{if}\;  i \in \sigma , \;\;  |x_i| =1 \; \mbox{if} \; i \notin \sigma \Big\}.
\end{equation}
So we obtain that 
$$  cat_{T^m}(\mathcal{Z}_K) \leq |S|  .$$ 
Note that the set
$$ \Big\{ (S^1)^{\sigma} \; : \; \sigma \in S \;  \Big\}$$ 
is the set of all minimal orbits of $\mathcal{Z}_K$ with respect to $T^m$-action. 
So we have 
$$  cat_{T^m}(\mathcal{Z}_K) \geq |S| .$$
\end{proof}

\begin{prop}\label{moment_gconn}
The moment angle complex $\mathcal{Z}_K$ is $(T^m)$-connected.
\end{prop}
\begin{proof}
Since each $(D^2, S^1)^{\sigma}$ is $(T^m)$-connected, the result 
follows from \eqref{mont_ang}.
\end{proof}

Let ${\bf F}[v_1, \ldots, v_m]$ be the graded polynomial algebra over a filed
$\bf{F}$ with $\deg{(v_j)} =2$ for $j \in \{1, \ldots, m\}$. Then the quotient
ring ${\bf F}(K) = {\bf F}[v_1, \ldots, v_m]/I_K$ is called the {\it Stanley-Reisner ring} of $K$ if $I_K$ is the homogeneous ideal generated by all square-free
monomials $v_{j_1} \cdots v_{j_r}$ ($j_1 < \cdots < j_r$) such that $\{j_1, 
\ldots, j_r\}$ is not a simplex in $K$. We denote the zero-divisor cup length
of the ring ${\rm Tor}_{{\bf F}[v_1, \ldots, v_m]}\big({\bf F}(K), {\bf F}\big)$ by $ZCL(K)$.

\begin{prop}
Let $\{\sigma_1, \ldots, \sigma_s\}$ be the maximal simplices of $K$ on $m$
vertices. Then 
\begin{equation}\label{inequ}
ZCL(K) \leq TC_{T^m,2}(\mathcal{Z}_K) \leq \sum_{i, j=1}^s (k_{ij} +1).
\end{equation} 
where $k_{ij} = \Big| ([m] - \sigma_i) \cap ([m] - \sigma_j) \Big|$.
\end{prop}

\begin{proof}
By  \cite[Theorem 7.6]{BP}, we have the following ring isomorphism
$$ H^{\ast}(\mathcal{Z}_K)  \cong {\rm Tor}_{{\bf F}[v_1, \ldots, v_m]}({\bf F}(K), {\bf F}) .$$
Therefore, the left inequality in \eqref{inequ} follows from \cite[Theorem 7]{Far}
and the natural fact that $TC(\mathcal{Z}_K) \leq TC_{T^m,2}(\mathcal{Z}_K)$.

By Proposition \ref{moment_gconn} and \ref{rel_tc_catg}, we have
$TC_{T^m, 2}(\mathcal{Z}_K) \leq cat_{T^m}(\mathcal{Z}_K \times \mathcal{Z}_K)$. 
Note that $$\{(D^2, S^1)^{\sigma_i} \times (D^2, S^2)^{\sigma_j} ~|~ i, j \in
\{1, \ldots, s\}\}$$ is $(T^m \times T^m)$-invariant (and hence $T^m$-invariant) open
cover of $\mathcal{Z}_K \times \mathcal{Z}_K$. Each open set $(D^2, S^1)^{\sigma_i}
\times (D^2, S^2)^{\sigma_j}$ is $T^m$-contractible to $(S^1)^{\sigma_i} \times
(S^1)^{\sigma_j}$ where $(S^1)^{\sigma}$ is defined in \eqref{s1sigma}. Since
$T^m$ acts on $(S^1)^{\sigma_i} \times (S^1)^{\sigma_j}$ diagonally as group
operation, then the orbit types are same and the 
corresponding orbit space $((S^1)^{\sigma_i} \times (S^1)^{\sigma_j})/T^m$ is
homeomorphic to $(S^1)^{k_{ij}}$ for $i, j \in \{1, \ldots, s\}$. So by
Lemma \ref{cat_gcat}, $$cat_{T^m}((S^1)^{\sigma_i} \times (S^1)^{\sigma_j})
=cat (S^1)^{k_{ij}} = k_{ij} +1$$ for $i, j \in \{1, \ldots, s\}$. Therefore
$cat_{T^m}(\mathcal{Z}_K \times \mathcal{Z}_K) \leq \displaystyle \sum_{i, j=1}^s (k_{ij}+1)$.
\end{proof}


{\bf Acknowledgement.} 
The research of the first author was supported by the University of Regina, the Atlantic Association for Research in the Mathematical Sciences (AARMS), the Natural Sciences and Engineering Research Council of Canada (NSERC) and by the Air Force Office of
Scientific Research, Air Force Material Command, USAF under Award
No. FA9550-15-1-0331.
The second author was supported by the Pacific Institute for Mathematical Sciences (PIMS), the University of Regina
and by the University of Calgary.



%

\end{document}